\documentclass[11pt,twoside]{article}
\usepackage{amssymb,mathrsfs}
\usepackage{amsfonts,amsmath,latexsym}
\usepackage{enumerate}

\usepackage{authblk}

\usepackage{indentfirst}
\usepackage[amsmath,thmmarks]{ntheorem}
\numberwithin{equation}{section}
\theoremheaderfont{\normalfont\rmfamily\bf}
\theorembodyfont{\normalfont\it }
\newtheorem{theorem}{\indent Theorem}[section]
\newtheorem{lemma}{\indent Lemma} [section]
\newtheorem{corollary}{\indent Corollary} [section]

\newtheorem{definition}{\indent Definition} [section]
\newtheorem{remark}{\indent Remark} [section]

  \theoremstyle{nonumberplain}

\theorembodyfont{\normalfont \rm }
 \theoremsymbol{$\square$}
  \newtheorem{proof}{\indent Proof}

\allowdisplaybreaks

\DeclareMathOperator*{\essinf}{ess\, inf}
\DeclareMathOperator*{\esssup}{ess\, sup}
\def\rr{\mathbb{R}}
\def\rn{{\rr}^n}
\def\loc{{\rm loc}}

\begin{document}

\title{\bf Some notes on commutators of the fractional
maximal function on variable Lebesgue spaces
   \footnotetext{E-mail:  puzhang@sohu.com (Pu Zhang);
     zengyan@hpu.edu.cn (Zengyan Si); jl-wu@163.com (J. L. Wu)}
     }

\author[1]{Pu Zhang\thanks{Corresponding author: Pu Zhang}}
\author[2]{Zengyan Si}
\author[1]{Jianglong Wu}

\affil[1]{\small\it Department of Mathematics, Mudanjiang Normal
          University, Mudanjiang 157011, P. R. China}

\affil[2]{\small\it School of Mathematics and Information Science,
Henan Polytechnic University,

         Jiaozuo 454000, P. R. China}

\date{ }
\maketitle

\begin{center}\begin{minipage}{14.5cm}

{\bf Abstract.} Let $0<\alpha<n$ and $M_{\alpha}$ be the fractional
maximal function. The nonlinear commutator of $M_{\alpha}$ and a
locally integrable function $b$ is given by
$[b,M_{\alpha}](f)=bM_{\alpha}(f)-M_{\alpha}(bf)$. In this paper,
we mainly give some necessary and sufficient conditions for
the boundedness of $[b,M_{\alpha}]$ on variable Lebesgue spaces when
$b$ belongs to Lipschitz or $BMO(\rn)$ spaces, by which some new
characterizations for certain subclasses of Lipschitz and $BMO(\rn)$
spaces are obtained.
% Moreover, some of our results are new even in the case of the Lebesgue spaces with constant exponents.



\end{minipage}
\end{center}

\medskip

\section{Introduction and Main Results}

Let $T$ be the classical singular integral operator. In 1976,
Coifman, Rochberg and Weiss \cite{crw} studied the commutator
generated by $T$ and a function $b\in{BMO(\rn)}$ as follows
\begin{equation} \label{equ.1.1}    %--------------equ.1.1--------------------
[b,T](f)(x)=T\big((b(x)-b(\cdot))f(\cdot)\big)(x)
  =b(x)T(f)(x)-T(bf)(x).
\end{equation}

A well-known result states that $[b,T]$ is bounded on $L^p(\rn)$ for
$1<p<\infty$ if and only if $b\in{BMO(\rn)}$. The sufficiency was
obtained by Coifman, Rochberg and Weiss \cite{crw} and the necessity
was proved by Janson \cite{j}. Moreover, Janson also gave some
characterizations of the Lipschitz space
${\dot{\Lambda}}_{\beta}(\rn)$ (see Definition \ref{def.lip} below)
via commutator $[b,T]$ in \cite{j} and proved that $[b,T]$ is
bounded from $L^p(\rn)$ to $L^q(\rn)$, for $1<p<n/\beta$,
$1/p-1/q=\beta/n$ and $0<\beta<1$, if and only if
$b\in{\dot{\Lambda}_{\beta}(\rn)}$ (see also Paluszy\'nski
\cite{p}).

As usual, a cube $Q\subset \rn$ always means its sides parallel to
the coordinate axes. Denote by $|Q|$ the Lebesgue measure and
$\chi_Q$ the characteristic function of $Q$. For $f\in
{L}_{\loc}^{1}(\rn)$, we write
$$f_Q=\frac1{|Q|}\int_Q f(x)dx.
$$

Let $0\le\alpha<n$ and $f\in {L^1_{\mathrm{loc}}(\rn)}$, the
fractional maximal function $M_{\alpha}$ is given by
$$M_{\alpha}(f)(x)=\sup_{Q\ni x} \frac{1}{|Q|^{1-{\alpha}/n}}
\int_{Q} |f(y)| dy,
$$
where the supremum is taken over all cubes $Q\subset\rn$ containing
$x$. When $\alpha=0$, we simply write $M$ instead of $M_{0}$, which
is exactly the Hardy-Littlewood maximal function.

Similar to (\ref{equ.1.1}), we can define two different kinds of
commutators of the fractional maximal function as follows.

\begin{definition}       \label{def.com}
Let $0\le\alpha<n$ and $b$ be a locally integrable function. The
maximal commutator of $M_{\alpha}$ and $b$ is given by
$$M_{\alpha,b}(f)(x)=\sup_{Q\ni x} \frac{1}{|Q|^{1-{\alpha}/n}}
\int_{Q} |b(x)-b(y)||f(y)| dy,
$$
where the supremum is taken over all cubes $Q\subset \rn$
containing $x$.

The nonlinear commutator of $M_{\alpha}$ and $b$ is defined by
$$[b,M_{\alpha}](f)(x)=b(x)M_{\alpha}(f)(x)-M_{\alpha}(bf)(x).
$$
\end{definition}

When $\alpha=0$, we simply denote by $M_b:=M_{0,b}$ and
$[b,M]:=[b,M_{0}]$.

We call $[b,M_{\alpha}]$ the nonlinear commutator because it is not
even a sublinear operator, although the commutator $[b,T]$ is a
linear one. We would like to remark that the nonlinear commutator
$[b,M_{\alpha}]$ and the maximal commutator $M_{\alpha,b}$
essentially differ from each other. For example, $M_{\alpha,b}$ is
positive and sublinear, but $[b,M_{\alpha}]$ is neither positive nor
sublinear.

The mapping property of $[b,M_{\alpha}]$ has been extensively
studied. See \cite{agkm, bmr, dgh, fj, gma, gdh, ms, xie, zhp1,
zhp2, zhw1, zhw2, zhw3, zws} for instance. There are some
applications of nonlinear commutators in Analysis. For example,
$[b,M]$ can be used in studying the products of functions in $H^1$
and $BMO$ (see \cite{bijz} for instance).

In 1990, by using the real interpolation techniques, Milman and
Schonbek \cite{ms} obtained a commutator result, by which they
obtained the $L^p$-boundedness of $[b,M]$ and
$[b,M_{\alpha}](0<\alpha<n)$ when $b\in{BMO(\rn)}$ and $b\ge0$. In
2000, Bastero, Milman and Ruiz \cite{bmr} considered the necessary
and sufficient conditions for the boundedness of $[b,M]$ in
$L^p(\rn)$ when $b$ belongs to $BMO(\rn)$. In 2009, Zhang and Wu
\cite{zhw1} extended their results to commutators of the fractional
maximal function. The results in \cite{bmr} and \cite{zhw1} were
extended to variable Lebesgue spaces in \cite{zhw2} and \cite{zhw3}.

Recently, Zhang \cite{zhp1} studied the commutator $[b,M]$ when $b$
belongs to Lipschitz spaces. Some necessary and sufficient
conditions for the boundedness of $[b,M]$ on Lebesgue and Morrey
spaces are given. Some of the results were extended to variable
Lebesgue spaces in \cite{zhp2} and to the context of Orlicz spaces
in \cite{gd}, \cite{gdh} and \cite{zws}.

Motivated by the papers mentioned above,
in this paper, we mainly study the mapping properties of
$[b,M_{\alpha}]$ in variable Lebesgue spaces when $b$ belongs to
Lipschitz or $BMO(\rn)$ spaces. More precisely, we will give some
new kind of necessary and sufficient conditions for the boundedness
of $[b,M_{\alpha}]$ on variable Lebesgue spaces, by which some new
characterizations for certain subclasses of Lipschitz and $BMO(\rn)$
spaces are obtained. Moreover, our results also give affirmative
answers to the questions mentioned in \cite{gdh} and \cite{zhw2}
(see Remark \ref{rem.cor.nc-lip} and Remark \ref{rem.thm.nc-bmo}
below, respectively). We would like to note that some of our results
are new even in the case of Lebesgue spaces with constant exponents.

To state the results, we first recall some definitions and
notations.

Let $\gamma\ge 0$, for a fixed cube $Q_0$,
the fractional maximal function with respect to $Q_0$ of a
locally integrable function $f$ is given by
$$M_{\gamma,Q_0}(f)(x) =\sup_{\substack{Q\ni x\\ Q\subseteq{Q_{0}}}}
\frac{1}{|Q|^{1-\gamma/n}}\int_{Q}|f(y)|dy,
$$
where the supremum is taken over all cubes $Q$ such that
$x \in {Q}\subseteq {Q_0}$.

When $\gamma=0$, we simply write $M_{Q_0}$ instead of $M_{0,Q_0}$.

\begin{definition}   \label{def.lip}
Let $0<\beta<1$, we say a function $b$ belongs to the Lipschitz
space $\dot{\Lambda}_{\beta}(\rn)$, denoted by
$b\in\dot{\Lambda}_{\beta}(\rn)$, if there exists a constant $C>0$
such that for all $x,y\in\rn$,
$$|b(x)-b(y)|\le {C}|x-y|^{\beta}.
$$
The smallest such constant $C$ is called the $\dot{\Lambda}_{\beta}$
norm of $b$ and is denoted by $\|b\|_{\dot{\Lambda}_{\beta}}$.
\end{definition}

\begin{definition}   \label{def.bmo}
A locally integrable function $f$ is said to belong to $BMO(\rn)$
if
$$\|f\|_{BMO}:=\sup_{Q}\frac1{|Q|}\int_Q|f(x)-f_Q|dx<\infty,
$$
where the supremum is taken over all cubes $Q$ in $\rn$.
\end{definition}

For a function $b$ defined on $\rn$, we denote by
\begin{displaymath}
b^{-}(x)= \left\{\begin{array}{ll}
0, &\hbox{if}\ b(x)\ge 0, \\
|b(x)|, &\hbox{if}\ b(x)< 0,
\end{array}\right.
\end{displaymath}
and $b^{+}(x)=|b(x)|-b^{-}(x)$. Obviously, $b^{+}(x)-b^{-}(x)=b(x)$.

%--------------------------------begin Definition variable Lebesgue space--------------------------------
\begin{definition} \label{def.var-lp}
Let $p(\cdot): \rn \to [1,\infty)$ be a measurable function. The
variable Lebesgue space, $L^{p(\cdot)}(\rn)$, is defined by
$$ L^{p(\cdot)}(\rn)=\bigg\{f~ \mbox{measurable}:
\int_{\rn} \bigg(\frac{|f(x)|}{\lambda}\bigg)^{p(x)} dx <\infty
~\mbox{for some constant}~ \lambda>0\bigg\}.
$$
\end{definition}
%--------------------------------end Definition variable Lebesgue space--------------------------------
The set $L^{p(\cdot)}(\rn)$ becomes a Banach space with respect to
the norm
\begin{equation*}
\|f\|_{L^{p(\cdot)}(\rn)}=\inf \bigg\{\lambda>0: \int_{\rn}
\bigg(\frac{|f(x)|}{\lambda} \bigg)^{p(x)}dx \le 1 \bigg\}.
\end{equation*}

We refer to \cite{cf}, \cite{dhhr}, \cite{kmrs1} and \cite{kmrs2}
for more details on function spaces with variable exponents.

Denote by $\mathscr{P}(\rn)$ the set of all measurable functions
$p(\cdot): \rn\to[1,\infty)$ such that
$$1< p_{-}:=\essinf_{x\in \rn}p(x) ~~\mathrm{and}~~
{p_{+}:}=\esssup_{x\in \rn}p(x)<\infty,
$$
and by $\mathscr{B}(\rn)$ the set of all $p(\cdot) \in
\mathscr{P}(\rn)$ such that $M$ is bounded on
$L^{p(\cdot)}(\rn)$.

\begin{remark} \label{rem.cfmp-1}
If $p(\cdot)\in \mathscr{B}(\rn)$ and $\lambda>1$, then
${\lambda} p(\cdot)\in \mathscr{B}(\rn)$.
See Remark 2.13 in \cite{cfmp}.
\end{remark}

For notational convenience, we introduce a notation
$\mathscr{B}^{\gamma}(\rn)$ as follows.

\begin{definition}
Let $0<\gamma<n$. We say an ordered pair of variable exponents
$\big(p(\cdot),q(\cdot)\big) \in \mathscr{B}^{\gamma}(\rn)$,
if $p(\cdot)\in \mathscr{P}(\rn)$ with $p_{+}<n/{\gamma}$ and
$1/q(\cdot)=1/p(\cdot)-\gamma/n$ with $q(\cdot)(n-\gamma)/n \in
\mathscr{B}(\rn)$.
\end{definition}

\begin{remark}       \label{rem.cfmp-2}   %------------------------Remark------------------------
The condition $q(\cdot)(n-\gamma)/n \in \mathscr{B}(\rn)$ is
equivalent to saying that there exists $q_0$ with $n/(n-\gamma)<
q_0<\infty$ such that $q(\cdot)/q_0 \in\mathscr{B}(\rn)$. Moreover,
$q(\cdot)(n-\gamma)/n \in \mathscr{B}(\rn)$ implies $q(\cdot) \in
\mathscr{B}(\rn)$. See Remark 2.13 in \cite{cfmp} for details.
\end{remark}

Our results can be stated as follows.

%--------------------------------begin Theorem nonlinear commutator with Lipschitz function on variable Lebesgue spaces--------------------------------
\begin{theorem}  \label{thm.nc-lip}
Let $0<\beta<1$, $0<\alpha<n$, $0<\alpha+\beta<n$ and $b$ be a
locally integrable function. Then the following statements are
equivalent:

(1) $b\in {\dot{\Lambda}_{\beta}(\rn)}$ and $b\ge 0$.

(2) $[b,M_{\alpha}]$ is bounded from $L^{p(\cdot)}(\rn)$ to
$L^{q(\cdot)}(\rn)$ for some
$\big(p(\cdot),q(\cdot)\big)\in\mathscr{B}^{\alpha+\beta}(\rn)$.

(3) $[b,M_{\alpha}]$ is bounded from $L^{p(\cdot)}(\rn)$ to
$L^{q(\cdot)}(\rn)$ for all
$\big(p(\cdot),q(\cdot)\big)\in\mathscr{B}^{\alpha+\beta}(\rn)$.

(4) There exists $s(\cdot)\in \mathscr{B}(\rn)$ such that
\begin{equation}         \label{equ.nc-lip}
\sup_{Q} \frac1{|Q|^{\beta/n}}
\frac{\big\|\big(b-|Q|^{-\alpha/n}M_{\alpha,Q}(b)\big)
\chi_{Q}\big\|_{L^{s(\cdot)}(\rn)}}{\|\chi_{Q}\|_{L^{s(\cdot)}(\rn)}}
<\infty.
\end{equation}

(5) For all $s(\cdot)\in \mathscr{B}(\rn)$ we have
(\ref{equ.nc-lip}).
\end{theorem}

\begin{remark}       \label{rem.thm.nc-lip}   %------------------------Remark------------------------
For the case $\alpha=0$, the result was proved in \cite{zhp2}.
Moreover, (\ref{equ.nc-lip}) gives a new characterization of
nonnegative Lipschitz functions, compaired with \cite[Theorem
1.5]{zhp2}.
\end{remark}

For the case $p(\cdot)$ and $q(\cdot)$ being constants, we have the
following results from Theorem \ref{thm.nc-lip}, which is new even
for this case.

%--------------------------------begin Theorem commutator with Lipschitz function--------------------------------
\begin{corollary}  \label{cor.nc-lip}
Let $0<\beta<1$, $0<\alpha<n$, $0<\alpha+\beta<n$ and $b$ be a
locally integrable function. Then the following statements are
equivalent:

(1) $b\in {\dot{\Lambda}_{\beta}(\rn)}$ and $b\ge 0$.

(2) $[b,M_{\alpha}]$ is bounded from $L^p(\rn)$ to $L^q(\rn)$
for some $p$ and $q$ such that $1<p<n/(\alpha+\beta)$ and
$1/q=1/p-(\alpha+\beta)/n$.

(3) $[b,M_{\alpha}]$ is bounded from $L^p(\rn)$ to $L^q(\rn)$
for all $p$ and $q$ such that $1<p<n/(\alpha+\beta)$ and
$1/q=1/p-(\alpha+\beta)/n$.

(4) There exists $s\in [1,\infty)$ such that
\begin{equation}         \label{equ.cor.nc-lip}
\sup_{Q} \frac1{|Q|^{\beta/n}} \bigg(\frac1{|Q|} \int_Q
\big|b(x)-|Q|^{-\alpha/n}M_{\alpha,Q}(b)(x)\big|^s dx \bigg)^{1/s}
<\infty.
\end{equation}

(5) For all $s\in [1,\infty)$ we have (\ref{equ.cor.nc-lip}).
\end{corollary}
%--------------------------------end Theorem commutator with Lipschitz function---------------------------------------

\begin{remark} \label{rem.cor.nc-lip} %---------------Remark--------------------
The result was proved for $\alpha=0$ in \cite[Theorem 1.4]{zhp1}.
Corollary \ref{cor.nc-lip} improves the result of
\cite[Corollary 4.15]{gdh} essentially and answers a question
asked in \cite[Remark 4.17]{gdh} affirmatively.
Moreover, it was proved in \cite[Theorem 1.4]{zhp1}, see also
Lemma \ref{lem.lip+} below, that
$b\in{\dot{\Lambda}_{\beta}(\rn)}$ and $b\ge 0$ if and only if
\begin{equation}         \label{equ.rem.nc-lip}
\sup_{Q} \frac1{|Q|^{\beta/n}} \bigg(\frac1{|Q|} \int_Q
\big|b(x)-M_{Q}(b)(x)\big|^s dx \bigg)^{1/s} <\infty.
\end{equation}
Compared with (\ref{equ.rem.nc-lip}), (\ref{equ.cor.nc-lip}) gives a new
characterization for nonnegative Lipschitz functions.
\end{remark}

%--------------------------------begin Theorem nonlinear commutator with BMO --------------------------------
\begin{theorem}  \label{thm.nc-bmo}
Let $0<\alpha<n$ and $b$ be a locally integrable function. Then the
following statements are equivalent:

(1) $b\in {BMO(\rn)}$ and $b^{-}\in{L^{\infty}(\rn)}$.

(2) $[b,M_{\alpha}]$ is bounded from $L^{p(\cdot)}(\rn)$ to
$L^{q(\cdot)}(\rn)$
for some $\big(p(\cdot),q(\cdot)\big)\in\mathscr{B}^{\alpha}(\rn)$.

(3) $[b,M_{\alpha}]$ is bounded from $L^{p(\cdot)}(\rn)$ to
$L^{q(\cdot)}(\rn)$
for all $\big(p(\cdot),q(\cdot)\big)\in\mathscr{B}^{\alpha}(\rn)$.

(4) There exists $s(\cdot)\in \mathscr{B}(\rn)$ such that
\begin{equation}         \label{equ.nc-bmo}
\sup_{Q} \frac{\big\|\big(b-|Q|^{-\alpha/n}M_{\alpha,Q}(b)\big)
\chi_{Q}\big\|_{L^{s(\cdot)}(\rn)}}{\|\chi_{Q}\|_{L^{s(\cdot)}(\rn)}}
<\infty.
\end{equation}

(5) For all $s(\cdot)\in \mathscr{B}(\rn)$ we have (\ref{equ.nc-bmo}).
\end{theorem}

\begin{remark}   \label{rem.thm.nc-bmo}
The equivalence of (1), (2) and (3) was proved in \cite{zhw2}.
Statements (4) and (5) give new necessary and sufficient condition for the
statements (1), (2) and (3). Especially, (\ref{equ.nc-bmo}) gives a
new characterization for $b\in {BMO(\rn)}$ and
$b^{-}\in{L^{\infty}(\rn)}$, which also answers a question asked in
\cite[Remark 4.1]{zhw2}. For the case $\alpha=0$, the result was
obtained in \cite{zhw3}.
\end{remark}

For the case $p(\cdot)$ and $q(\cdot)$ being constants, we have the
following results by Theorem \ref{thm.nc-bmo}.

\begin{corollary}  \label{cor.nc-bmo}
Let $0<\alpha<n$ and $b$ be a locally integrable function.
Then the following statements are equivalent:

(1) $b\in BMO(\rn)$ and $b^{-}\in L^{\infty}(\rn)$.

(2) $[b,M_{\alpha}]$ is bounded from $L^p(\rn)$ to $L^q(\rn)$ for
some $p$ and $q$ such that $1<p<n/\alpha$ and $1/q=1/p-\alpha/n$.

(3) $[b,M_{\alpha}]$ is bounded from $L^p(\rn)$ to $L^q(\rn)$ for
all $p$ and $q$ such that $1<p<n/\alpha$ and $1/q=1/p-\alpha/n$.

(4) There exists $s\in [1,\infty)$ such that
\begin{equation}         \label{equ.cor.nc-bmo}
\sup_{Q} \bigg(\frac1{|Q|} \int_Q
|b(x)-|Q|^{-\alpha/n}M_{\alpha,Q}(b)(x)|^s dx \bigg)^{1/s} <\infty.
\end{equation}

(5) For all $s\in [1,\infty)$ we have (\ref{equ.cor.nc-bmo}).
\end{corollary}
%--------------------------------end Theorem commutator with BMO ---------------------------------------

\begin{remark}        \label{rem.cor.nc-bmo}  %-----------Remark----------
It was shown in \cite{bmr} and \cite{zhw1} that statements (1), (2)
and (3) are equivalent to
\begin{equation}         \label{equ.rem.cor.nc-bmo}
\sup_{Q} \frac1{|Q|} \int_Q|b(x)-M_{Q}(b)(x)|^sdx<\infty,
\end{equation}
respectively. Compared with (\ref{equ.rem.cor.nc-bmo}),
(\ref{equ.cor.nc-bmo}) gives a new characterization.
\end{remark}

Next, we give some necessary and sufficient conditions for the
boundedness of the maximal commutator $M_{\alpha,b}$ on variable
Lebegue spaces when $b$ belongs to Lipschitz space.

%--------------------------------begin maximal commutator with Lipschitz function on variable Lebesgue spaces-------------------------------
\begin{theorem}        \label{thm.mc-lip} %----------------------------------------Theorem   -------------
Let $0<\beta<1$, $0<\alpha<n$, $0<\alpha+\beta<n$ and $b$ be a
locally integrable function. Then the following statements are
equivalent:

(1) $b\in {\dot{\Lambda}_{\beta}(\rn)}$.

(2) $M_{\alpha,b}$ is bounded from $L^{p(\cdot)}(\rn)$ to
$L^{q(\cdot)}(\rn)$ for some
$\big(p(\cdot),q(\cdot)\big)\in\mathscr{B}^{\alpha+\beta}(\rn)$.

(3) $M_{\alpha,b}$ is bounded from $L^{p(\cdot)}(\rn)$ to
$L^{q(\cdot)}(\rn)$ for all
$\big(p(\cdot),q(\cdot)\big)\in\mathscr{B}^{\alpha+\beta}(\rn)$.

(4) There exists $s(\cdot)\in \mathscr{B}(\rn)$ such that
\begin{equation}             \label{equ.mc-lip}
\sup_{Q} \frac1{|Q|^{\beta/n}} \frac{\|(b-b_{Q})
\chi_{Q}\|_{L^{s(\cdot)}(\rn)}} {\|\chi_{Q}\|_{L^{s(\cdot)}(\rn)}}
<\infty.
\end{equation}

(5) For all $s(\cdot)\in \mathscr{B}(\rn)$ we have (\ref{equ.mc-lip}).
\end{theorem}
%--------------------------------end maximal commutator with Lipschitz function on variable Lebesgue spaces  --------------------------------

\begin{remark} \label{rem.thm.mc-lip}   %------------------------Remark------------------------
For the case $\alpha=0$, similar results were given in \cite{zhp1}
for Lebesgue spaces with constant exponents and in \cite{zhp2} for
the variable case.
\end{remark}

When $p(\cdot)$ and $q(\cdot)$ are constants, we get the following
results from Theorem \ref{thm.mc-lip}.

\begin{corollary}  \label{cor.mc-lip}
Let $0<\beta<1$, $0<\alpha<n$, $0<\alpha+\beta<n$ and $b$ be a
locally integrable function. Then the following statements are
equivalent:

(1) $b\in {\dot{\Lambda}_{\beta}(\rn)}$.

(2) $M_{\alpha,b}$ is bounded from $L^p(\rn)$ to $L^q(\rn)$ for some
$p$ and $q$ such that $1<p<n/(\alpha+\beta)$ and
$1/q=1/p-(\alpha+\beta)/n$.

(3) $M_{\alpha,b}$ is bounded from $L^p(\rn)$ to $L^q(\rn)$ for all
$p$ and $q$ such that $1<p<n/(\alpha+\beta)$ and
$1/q=1/p-(\alpha+\beta)/n$.

(4) There exists $s\in [1,\infty)$ such that
\begin{equation}             \label{equ.cor.mc-lip}
\sup_Q \frac1{|Q|^{\beta/n}} \bigg(\frac1{|Q|}
\int_Q|b(x)-b_Q|^sdx\bigg)^{1/s}<\infty
\end{equation}

(5) For all $s\in [1,\infty)$ we have (\ref{equ.cor.mc-lip}).
\end{corollary}

\begin{remark} \label{rem.cor.mc-lip}   %------------------------Remark------------------------
The equivalence of (1), (2) and (3) was proved in \cite{zhp1} (for
$\alpha=0$) and in \cite{gdh} (for $0<\alpha<n$). The equivalence of
(1), (4) and (5) is contained in Lemma \ref{lem.lip} below.
\end{remark}

Finally, for the case of completeness of this paper, we state a
result similar to Theorem \ref{thm.mc-lip} without proof, which can
be deduced from \cite{zhw2} and \cite{izuki2}.

%--------------------------------begin maximal commutator with BMO function on variable Lebesgue spaces-------------------------------
\begin{theorem}        \label{thm.mc-bmo} %----------------------------------------Theorem   -------------
Let $0<\alpha<n$ and $b$ be a locally integrable function. Then the
following statements are equivalent:

(1) $b\in {BMO(\rn)}$.

(2) $M_{\alpha,b}$ is bounded from $L^{p(\cdot)}(\rn)$ to
$L^{q(\cdot)}(\rn)$ for some
$\big(p(\cdot),q(\cdot)\big)\in\mathscr{B}^{\alpha}(\rn)$.

(3) $M_{\alpha,b}$ is bounded from $L^{p(\cdot)}(\rn)$ to
$L^{q(\cdot)}(\rn)$ for all
$\big(p(\cdot),q(\cdot)\big)\in\mathscr{B}^{\alpha}(\rn)$.

(4) There exists $s(\cdot)\in \mathscr{B}(\rn)$ such that
\begin{equation}             \label{equ.mc-bmo}
\sup_{Q} \frac{\|(b-b_{Q}) \chi_{Q}\|_{L^{s(\cdot)}(\rn)}}
{\|\chi_{Q}\|_{L^{s(\cdot)}(\rn)}} <\infty.
\end{equation}

(5) For all $s(\cdot)\in \mathscr{B}(\rn)$ we have
(\ref{equ.mc-bmo}).
\end{theorem}

\begin{remark} \label{rem.thm.mc-bmo}   %------------------------Remark------------------------
We note that Theorem \ref{thm.mc-bmo} follows from
\cite{zhw2} and \cite{izuki2} directly. Indeed, the equivalence of
(1), (2) and (3) was proved in \cite{zhw2} Theorems 3.1 and 3.2
for $\alpha=0$ and $0<\alpha<n$, respectively, and
the equivalence of (1), (4) and (5) was obtained in \cite[Lemma
3]{izuki2}.
\end{remark}

If $p(\cdot)$ and $q(\cdot)$ are constants, we have a result similar
to Corollary \ref{cor.mc-lip}. We omit the details.

The remainder of this paper is organized as follows. In the next
section, we give some lemmas that will be used later. In Section 3,
we prove Theorems \ref{thm.nc-lip}, \ref{thm.nc-bmo} and
\ref{thm.mc-lip}.

%----------------------------------------------------------------Section Preliminaries------------------------------------
\section{Preliminaries and Lemmas}         \label{preliminaries}

It is known that Lipschitz space $\dot{\Lambda}_{\beta}(\rn)$
coincides with some Morrey-Companato space (see, e.g., \cite{jtw})
and can be characterized by mean oscillation as the following lemma,
which is due to DeVore and Sharpley \cite{ds} and Janson, Taibleson
and Weiss \cite{jtw} (see also Paluszy\'nski \cite{p}).

%---------------------------------------- --------------------------------
\begin{lemma}    \label{lem.lip}
Let $0<\beta<1$ and $1\le {q}<\infty$. Define
$$\dot{\Lambda}_{\beta,q}(\rn):=\bigg\{f \in{L_{\rm{loc}}^1(\rn)}:
\|f\|_{\dot{\Lambda}_{\beta,q}} = \sup_Q \frac1{|Q|^{\beta/n}}
\bigg(\frac1{|Q|} \int_Q|f(x)-f_Q|^qdx\bigg)^{1/q}<\infty \bigg\}.
$$
Then, for all $0<\beta<1$ and $1\le {q}<\infty$,
$\dot{\Lambda}_{\beta}(\rn)=\dot{\Lambda}_{\beta,q}(\rn)$
with equivalent norms.
\end{lemma}

From the proof of Theorem 1.4 in \cite{zhp1}, we can obtain the
following characterization of nonnegative Lipschitz functions.

%----------------------------------------Characterization of nonnegative Lipschitz function cite{zhp1}--------------------------------
\begin{lemma}   \label{lem.lip+}
Let $0<\beta<1$ and $b$ be a locally integrable function. Then the
following statements are equivalent:

(1) If $b\in {\dot{\Lambda}_{\beta}(\rn)}$ and $b\ge 0$.

(2) For all $1\le {s}<\infty$,
\begin{equation}    \label{equ.lem.lip+}
\sup_{Q} \frac1{|Q|^{\beta/n}} \bigg(\frac1{|Q|} \int_Q
|b(x)-M_{Q}(b)(x)|^s dx \bigg)^{1/s} <\infty.
\end{equation}

(3) (\ref{equ.lem.lip+}) holds for some $1\le {s}<\infty$.
\end{lemma}

\begin{proof}
Since the implication (2)$\Rightarrow$(3) follows readily and
the implication (3)$\Rightarrow$(1) was proved in \cite[Theorem 1.4]{zhp1},
we only need to prove (1)$\Rightarrow$(2).

If $b\in {\dot{\Lambda}_{\beta}(\rn)}$ and $b\ge 0$, then it follows
from \cite[Theorem 1.4]{zhp1} that (\ref{equ.lem.lip+}) holds for
all $s$ with $n/(n-\beta)<s<\infty$. Applying H\"older's inequality
we see that (\ref{equ.lem.lip+}) also holds for $1\le {s}\le
{n/(n-\beta)}$. So, the implication $(1)\Rightarrow(2)$ is proven.
\end{proof}

\begin{lemma}[\cite{bmr}] \label{lem.bmr}
Let $b$ be a locally integrable function. Then the following
statements are equivalent:

(1) $b\in {BMO(\rn)}$ and $b^{-} \in {L^{\infty}(\rn)}$.

(2) There exists $s\in[1,\infty)$ such that
\begin{equation}         \label{equ.lem.bmr}
\sup_{Q} \frac1{|Q|} \int_Q\big|b(x)-M_{Q}(b)(x)\big|^sdx <\infty.
\end{equation}

(3) For all $s\in[1,\infty)$ we have (\ref{equ.lem.bmr}).
\end{lemma}

The following strong-type estimates for the fractional maximal
function is well known, see \cite{d} or \cite{g} for details.

%---------------------------------------- --------------------------------
\begin{lemma}    \label{lem.fracmax}
Let $0<\gamma<n$, $1<p<{n/\gamma}$ and $1/q=1/p-\gamma/n$. Then
there exists a positive constant $C(n,\gamma,p)$ such that
$$\|{M}_{\gamma}(f)\|_{L^q(\rn)}\le {C(n,\gamma,p)}\|f\|_{L^p(\rn)}.
$$
\end{lemma}

As for the boundedness of the fractional maximal function on
variable Lebesgue spaces, the following result was given in
\cite{cfmp}. See Corollary 2.12 and Remark 2.13 in \cite{cfmp} for
details.

%--------------------------------begin Lemma--------------------------------
\begin{lemma}       \label{lem.fracmax-var}
Let $0<\gamma<n$, $p(\cdot)\in \mathscr{P}(\rn)$ with
$p_{+}<n/\gamma$ and $1/q(\cdot)=1/p(\cdot)-\gamma/n$. If
$q(\cdot)(n-\gamma)/n \in \mathscr{B}(\rn)$, then ${M}_{\gamma}$ is
bounded from $L^{p(\cdot)}(\rn)$ to $L^{q(\cdot)}(\rn)$.
\end{lemma}

By Lemma \ref{lem.fracmax}, if $0<\gamma<n$, $1<p<{n/\gamma}$ and
$f\in{L^p(\rn)}$, then ${M}_{\gamma}(f)(x)<\infty$ almost
everywhere. A similar result is also valid in variable Lebesgue
spaces.

%--------------------------------begin Lemma--------------------------------
\begin{lemma}             \label{lem.fracmax-ae}
Let $0<\gamma<n$, $p(\cdot)\in \mathscr{P}(\rn)$ and
$1<p_{-}\le{p_{+}}<n/\gamma$. If $f \in {L^{p(\cdot)}(\rn)}$,
then ${M}_{\gamma}(f)(x)<\infty$ for almost every $x\in\rn$.
\end{lemma}

\begin{proof} Following the same procedure of the proof of
\cite[Proposition 3.15]{cf}, we can achieve the desired result.
Indeed, for any  $f \in {L^{p(\cdot)}(\rn)}$, by Theorem 2.51 in
\cite{cf} we can write $f=f_1+f_2$, where $f_1\in{L^{p_{+}}(\rn)}$
and $f_2\in{L^{p_{-}}(\rn)}$. Then ${M}_{\gamma}(f)(x)\le
{M}_{\gamma}(f_1)(x) +{M}_{\gamma}(f_2)(x)$. Noting that $1<p_{-}\le
p_{+}<n/\gamma$ and $0<\gamma<n$, by Lemma \ref{lem.fracmax} we see
that ${M}_{\gamma}(f_1)(x)$ and ${M}_{\gamma}(f_2)(x)$ are finite
almost everywhere. Then ${M}_{\gamma}(f)(x)<\infty$ for almost every
$x\in\rn$.
\end{proof}

We also need some basic properties of variable Lebesgue spaces.
Denoted by $p'(\cdot)$ the conjugate index of $p(\cdot)$. Obviously,
if $p(\cdot)\in\mathscr{P}(\rn)$ then $p(\cdot)\in\mathscr{P}(\rn)$.
The following lemma is known as the generalized H\"older's
inequality in variable Lebesgue spaces. See \cite{cf} and
\cite{dhhr} for details.

\begin{lemma}    \label{lem.holder}
(i) Let $p(\cdot)\in \mathscr{P}(\rn)$. Then there exists a positive
constant $C$ such that for all $f\in L^{p(\cdot)}(\rn)$ and $g\in
{L^{p'(\cdot)}(\rn)}$,
$$\int_{\rn} |f(x)g(x)|dx \le {C} \|f\|_{L^{p(\cdot)}(\rn)}
\|g\|_{L^{p'(\cdot)}(\rn)}.
$$

(ii) Let $p(\cdot), p_1(\cdot), p_2(\cdot)\in \mathscr{P}(\rn)$ and
$1/p(\cdot)=1/p_1(\cdot)+1/p_2(\cdot)$. Then there exists a positive
constant $C$ such that for all $f\in L^{p_1(\cdot)}(\rn)$ and $g\in
{L^{p_2(\cdot)}(\rn)}$,
$$\|fg\|_{L^{p(\cdot)}(\rn)} \le {C} \|f\|_{L^{p_1(\cdot)}(\rn)}
\|g\|_{L^{p_2(\cdot)}(\rn)}.
$$
\end{lemma}

\begin{lemma}[\cite{cu-w}] \label{lem.cu-w}
Given $p(\cdot)\in \mathscr{P}(\rn)$, then for all $r>0$ we have
$$\big\||f|^r\big\|_{L^{p(\cdot)}(\rn)}
= \|f\|^r_{L^{rp(\cdot)}(\rn)}.
$$
\end{lemma}

\begin{lemma}[\cite{izuki}] \label{lem.izuki}
Let $p(\cdot)\in \mathscr{B}(\rn)$, then there exists a constant
$C>0$ such that
$$\frac{1}{|Q|} \|\chi_{Q}\|_{L^{p(\cdot)}(\rn)}
\|\chi_{Q}\|_{L^{p'(\cdot)}(\rn)} \le C
$$
for all cubes $Q$ in $\rn$.
\end{lemma}

\begin{lemma}[\cite{zhp2}]  \label{lem.zhp2}
Let $0<\gamma<n$, $p(\cdot)\in \mathscr{P}(\rn)$ with
$p_{+}<n/\gamma$ and $1/q(\cdot)=1/p(\cdot)-\gamma/n$. If
$q(\cdot)(n-\gamma)/n \in \mathscr{B}(\rn)$, then there exists a
constant $C>0$ such that
$$\|\chi_Q\|_{L^{p(\cdot)}(\rn)}
\le {C}|Q|^{\gamma/n}\|\chi_Q\|_{L^{q(\cdot)}(\rn)}
$$
for all cubes $Q$ in $\rn$.
\end{lemma}
%--------------------------------end Lemma--------------------------------

Now, we give the following pointwise estimates for $[b,M_{\alpha}]$
when $b\in {\dot{\Lambda}_{\beta}(\rn)}$.

%----------------------------------------------------------------
\begin{lemma}  \label{lem.nc-m}
Let $0\le\alpha<n$, $0<\beta<1$, $0<\alpha+\beta<n$ and $f$ be a
locally integrable function. If $b\in {\dot{\Lambda}_{\beta}(\rn)}$
and $b\ge 0$, then, for any $x\in\rn$ such that
$M_{\alpha}(f)(x)<\infty$, we have
$$   % \label{equ.lem.nc-m}
\big|[b,M_{\alpha}]f(x)\big| \le
   \|b\|_{\dot{\Lambda}_{\beta}} M_{\alpha+\beta}(f)(x).
$$
\end{lemma}
%----------------------------------------------------------------

\begin{proof}
For any fixed $x\in\rn$ such that $M_{\alpha}(f)(x)<\infty$, if
$b\ge0$ and $b\in {\dot{\Lambda}_{\beta}(\rn)}$ then
\begin{align*}
|[b,M_{\alpha}](f)(x)|&=|b(x)M_{\alpha}(f)(x)-M_{\alpha}(bf)(x)|\\
 & =\bigg|\sup_{Q\ni{x}} \frac1{|Q|^{1-\alpha/n}} \int_Q
    b(x)|f(y)|dy-\sup_{Q\ni{x}} \frac1{|Q|^{1-\alpha/n}}
     \int_Q b(y)|f(y)|dy\bigg|\\
 &\le \sup_{Q\ni{x}}\frac1{|Q|^{1-\alpha/n}}
   \int_Q |b(x)-b(y)||f(y)|dy\\
 &\le \|b\|_{\dot{\Lambda}_{\beta}(\rn)}
     \sup_{Q\ni{x}}\frac1{|Q|^{1-(\alpha+\beta)/n}}\int_Q |f(y)|dy\\
      &\le \|b\|_{\dot{\Lambda}_{\beta}} M_{\alpha+\beta}(f)(x).
\end{align*}
\end{proof}

Finally, we also need the following result.

%----------------------------------------------------------------
\begin{lemma}[\cite{bmr}, \cite{zhw1}]  \label{lem.cube}
Let $0\le\gamma<n$, $Q$ be a cube in $\rn$ and $f$ be a locally
integrable function. Then for all $x\in{Q}$,
$$M_{\gamma}(f\chi_Q)(x)= M_{\gamma,Q}(f)(x)
$$
and
$$M_{\gamma}(\chi_Q)(x)= M_{\gamma,Q}(\chi_Q)(x)=|Q|^{\gamma/n}.
$$
\end{lemma}

%%%%%%%%%-------------------------------------------------------Section------------------------------------
\section{Proofs of Theorems \ref{thm.nc-lip},
\ref{thm.nc-bmo} and \ref{thm.mc-lip}}

To prove Theorem \ref{thm.nc-lip}, we first prove the following
lemma.

\begin{lemma}    \label{lem.nc-lip}
Let $0<\beta<1$ and $0<\gamma<n$. If $b$ is a locally integrable
function and satisfies
\begin{equation}         \label{equ.lem.nc-lip}
\sup_{Q} \frac1{|Q|^{\beta/n}}
\frac{\big\|\big(b-|Q|^{-\gamma/n}M_{\gamma,Q}(b)\big)
\chi_{Q}\big\|_{L^{s(\cdot)}(\rn)}}{\|\chi_{Q}\|_{L^{s(\cdot)}(\rn)}}
\le{C}
\end{equation}
for some $s(\cdot)\in \mathscr{B}(\rn)$, then
$b\in\dot{\Lambda}_{\beta}(\rn)$.
\end{lemma}

\begin{proof}
Some ideas are taken from \cite{bmr}, \cite{zhw1} and \cite{zhw2}.
Reasoning as the proof of (4.4) in \cite{zhw2}, see also the proof of
Lemma 2.4 in \cite{zhw1}, we have, for any cube $Q$,
\begin{align*}
\frac1{|Q|^{1+\beta/n}} \int_{Q}|b(x)-b_{Q}|dx
 \le \frac{2}{|Q|^{1+\beta/n}} \int_{Q}
  \big|b(x)-|Q|^{-\gamma/n}M_{\gamma,Q}(b)(x)\big|dx.
\end{align*}

Indeed, for any cube $Q$, let $E=\{x\in Q: b(x)\le b_{Q}\}$ and $F=\{x\in Q:
b(x)> b_{Q}\}$. It is easy to check that the following equality is
true (see \cite{bmr} page 3331):
$$\int_E |b(x)-b_{Q}|dx =\int_F |b(x)-b_{Q}|dx.
$$

Noticing the obvious estimate
$$|b_{Q}|\le|Q|^{-\gamma/n}M_{\gamma,Q}(b)(x)~~\hbox{for~any~}x\in{Q}
$$
and $b(x)\le{b_Q}$ for any $ x\in E$, we have
$$b(x)\le {b_{Q}} \le |b_{Q}| \le
|Q|^{-\gamma/n}M_{\gamma,Q}(b)(x)~~ \hbox{for any} ~x\in {E}.
$$
Then, for any $x\in {E}$,
$$|b(x)- b_{Q}|\le
\big|b(x)-|Q|^{-\gamma/n}M_{\gamma,Q}(b)(x)\big|.
$$

Therefore,
\begin{align*}
\frac1{|Q|^{1+\beta/n}} \int_{Q}|b(x)-b_{Q}|dx
 &=\frac1{|Q|^{1+\beta/n}} \int_{E\cup{F}}\big|b(x)-b_{Q}\big|dx \\
 &=\frac{2}{|Q|^{1+\beta/n}}\int_E |b(x)-b_{Q}|dx\\
 &\le \frac{2}{|Q|^{1+\beta/n}} \int_{E}
  \big|b(x)-|Q|^{-\gamma/n}M_{\gamma,Q}(b)(x)\big|dx\\
 &\le \frac{2}{|Q|^{1+\beta/n}} \int_{Q}
  \big|b(x)-|Q|^{-\gamma/n}M_{\gamma,Q}(b)(x)\big|dx.
\end{align*}

By Lemma \ref{lem.holder} (i), (\ref{equ.lem.nc-lip}) and
Lemma \ref{lem.izuki}, we get
\begin{equation*}
\begin{split}
&\frac1{|Q|^{1+\beta/n}} \int_{Q}|b(x)-b_{Q}|dx\\
 & \le \frac{C}{|Q|^{1+\beta/n}}
  \big\|\big(b-|Q|^{-\gamma/n}M_{\gamma,Q}(b)\big)
   \chi_{Q}\big\|_{L^{s(\cdot)}(\rn)}\|\chi_{Q}\|_{L^{s'(\cdot)}(\rn)}\\
& \le \frac{C}{|Q|} \|\chi_{Q}\|_{L^{s(\cdot)}(\rn)}
\|\chi_{Q}\|_{L^{s'(\cdot)}(\rn)}\\
 &  \le  C.
\end{split}
\end{equation*}

So, the proof is completed by applying Lemma \ref{lem.lip}.
\end{proof}

%--------------------------------------------proof of Theorem \ref{thm.nc-lip}------------------------
\begin{proof} {\bf of Theorem \ref{thm.nc-lip}}~Since the
implications $(3)\Rightarrow(2)$ and $(5)\Rightarrow(4)$ follows
readily, we only need to prove $(1)\Rightarrow(3)$,
$(2)\Rightarrow(4)$, $(4)\Rightarrow(1)$ and $(3)\Rightarrow(5)$.

$(1)\Rightarrow(3)$. Let $b\in {\dot{\Lambda}_{\beta}(\rn)}$ and
$b\ge 0$, we need to prove $[b,M_{\alpha}]$ is bounded from
$L^{p(\cdot)}(\rn)$ to $L^{q(\cdot)}(\rn)$ for all
$\big(p(\cdot),q(\cdot)\big)\in\mathscr{B}^{\alpha+\beta}(\rn)$. For
such $p(\cdot)$ and any $f\in{L^{p(\cdot)}(\rn)}$, it follow from
Lemma \ref{lem.fracmax-ae} that $M_{\alpha}(f)(x)<\infty$ for almost
every $x\in\rn$. By Lemma \ref{lem.nc-m} we have
$$|[b,M_{\alpha}]f(x)|\le\|b\|_{\dot{\Lambda}_{\beta}}M_{\alpha+\beta}(f)(x).
$$
Then, statement (3) follows from Lemma \ref{lem.fracmax-var}.

$(2)\Rightarrow(4)$. Let $\big(p(\cdot),q(\cdot)\big)\in
\mathscr{B}^{\alpha+\beta}(\rn)$ such that $[b,M_{\alpha}]$ is
bounded from $L^{p(\cdot)}(\rn)$ to $L^{q(\cdot)}(\rn)$, we will
verify (\ref{equ.nc-lip}) for $s(\cdot)=q(\cdot)$. For any fixed
cube $Q$ and any $x\in{Q}$, it follows from Lemma \ref{lem.cube}
that
$$M_{\alpha}(\chi_{Q})(x) =M_{\alpha,Q}(\chi_Q)(x)=|Q|^{\alpha/n}
~~\hbox{and}~~ M_{\alpha}(b\chi_{Q})(x) =M_{\alpha,Q}(b)(x).
$$

Then, for any $x\in{Q}$,
\begin{align*}
&b(x)-|Q|^{-\alpha/n} M_{\alpha,Q}(b)(x)\\
 &=|Q|^{-\alpha/n}\big(b(x)|Q|^{\alpha/n}
   -M_{\alpha,Q}(b)(x)\big)\\
 &=|Q|^{-\alpha/n} \big(b(x)M_{\alpha}(\chi_{Q})(x)
   -M_{\alpha}(b\chi_{Q})(x)\big)\\
 &=|Q|^{-\alpha/n}[{b},M_{\alpha}](\chi_{Q})(x).
\end{align*}
Thus
$$\big(b(x)-|Q|^{-\alpha/n} M_{\alpha,Q}(b)(x)\big)\chi_Q(x)
  =|Q|^{-\alpha/n}[{b},M_{\alpha}](\chi_{Q})(x)\chi_Q(x).
$$

Noting that $[{b},M_{\alpha}]$ is bounded from $L^{p(\cdot)}(\rn)$
to $L^{q(\cdot)}(\rn)$ with $1/q(\cdot)=1/p(\cdot)-(\alpha+\beta)/n$
and applying Lemma \ref{lem.zhp2} we have
\begin{align*}
&\big\|\big(b-|Q|^{-\alpha/n}
    M_{\alpha,Q}(b)\big)\chi_Q\big\|_{L^{q(\cdot)}(\rn)}\\
 &\le|Q|^{-\alpha/n}\big\|[b,M_{\alpha}](\chi_{Q})\big\|_{L^{q(\cdot)}(\rn)}\\
 & \le{C}|Q|^{-\alpha/n}\big\|[b,M_{\alpha}]
    \big\|_{L^{p(\cdot)}\to{L^{q(\cdot)}}}
     \big\|\chi_{Q}\big\|_{L^{p(\cdot)}(\rn)}\\
 & \le{C} |Q|^{\beta/n}\big\|[b,M_{\alpha}]
  \big\|_{L^{p(\cdot)}\to{L^{q(\cdot)}}}
   \big\|\chi_{Q}\big\|_{L^{q(\cdot)}(\rn)},
\end{align*}
which gives (\ref{equ.nc-lip}) for $s(\cdot)=q(\cdot)$ since $Q$
is arbitrary and $C$ is independent of $Q$.

$(4)\Rightarrow(1)$. By Lemma \ref{lem.lip+}, it suffices to prove
\begin{equation}      \label{equ.nc-lip-1} %-------------------------------------------------
\sup_{Q}\frac1{|Q|^{1+\beta/n}}\int_Q|b(x)-M_{Q}(b)(x)|dx<\infty.
\end{equation}

For any fixed cube $Q$,
\begin{equation}      \label{equ.nc-lip-2} %-------------------------------------------------
\begin{split}
& \frac1{|Q|^{1+\beta/n}} \int_{Q} \big|b(x)-M_{Q}(b)(x)\big| dx \\
 &\le \frac1{|Q|^{1+\beta/n}} \int_{Q}
   \big|b(x)-|Q|^{-\alpha/n}M_{\alpha,Q}(b)(x)\big|dx\\
 &\qquad + \frac1{|Q|^{1+\beta/n}} \int_{Q}
  \big||Q|^{-\alpha/n}M_{\alpha,Q}(b)(x)-M_{Q}(b)(x)\big|dx\\
 &:=I_1+I_2.
  \end{split}
\end{equation}

For $I_1$, by statement (4) and applying Lemma \ref{lem.holder} (i)
and Lemma \ref{lem.izuki}, we have
\begin{align*}
I_1&\le\frac{C}{|Q|^{1+\beta/n}}
  \big\|\big(b-|Q|^{-\alpha/n}M_{\alpha,Q}(b)\big)\chi_Q\big\|_{L^{s(\cdot)}(\rn)}
  \|\chi_Q\|_{L^{s'(\cdot)}(\rn)}\\
 &\le\frac{C}{|Q|^{\beta/n}}
  \frac{\big\|\big(b-|Q|^{-\alpha/n}M_{\alpha,Q}(b)\big)\chi_Q\big\|_{L^{s(\cdot)}(\rn)}}
   {\|\chi_Q\|_{L^{s(\cdot)}(\rn)}}\\
 &\le{C},
\end{align*}
where the constant $C$ is independent of $Q$.

Next, we consider $I_2$. Similar to the ones in the proof of Theorem 1.1
in \cite{zws}, we can get $I_2\le C$. Now, we give the proof of it.
For all $x\in{Q}$, it follows from Lemma \ref{lem.cube} that
$$M(\chi_Q)(x)=\chi_Q(x)=1 ~~\hbox{and}~~ M(b\chi_Q)(x)=M_Q(b)(x),
$$
and
$$M_{\alpha}(\chi_{Q})(x)=|Q|^{\alpha/n} ~~\hbox{and}~~
    M_{\alpha}(b\chi_{Q})(x) =M_{\alpha,Q}(b)(x).
$$

Then, for any $x\in{Q}$,
\begin{equation}           \label{equ.nc-lip-3} %-------------------------------------------------
\begin{split}
&\big||Q|^{-\alpha/n}M_{\alpha,Q}(b)(x)-M_{Q}(b)(x)\big|\\
 &\le|Q|^{-\alpha/n}\big|M_{\alpha,Q}(b)(x)-|Q|^{\alpha/n}|b(x)|\big|
  +\big||b(x)|- M_{Q}(b)(x)\big|\\
 &\le |Q|^{-\alpha/n}\big|M_{\alpha}(b\chi_{Q})(x)
   -|b(x)|M_{\alpha}(\chi_{Q})(x)\big|\\
  &\qquad +\big||b(x)|M(\chi_{Q})(x)-M(b\chi_{Q})(x)\big|\\
 &\le |Q|^{-\alpha/n}\big|[|b|,M_{\alpha}](\chi_{Q})(x)\big|
    +\big|[|b|,M](\chi_{Q})(x)\big|.
\end{split}
\end{equation}

Since $s(\cdot)\in\mathscr{B}(\rn)$, then statement (4) along with
Lemma \ref{lem.nc-lip} gives $b\in{\dot{\Lambda}_{\beta}(\rn)}$,
which implies $|b|\in {\dot{\Lambda}_{\beta}(\rn)}$. Thus, we can
apply Lemma \ref{lem.nc-m} to $[|b|,M_{\alpha}]$ and $[|b|,M]$ since
$|b|\in {\dot{\Lambda}_{\beta}(\rn)}$ and $|b|\ge0$. By Lemmas
\ref{lem.nc-m} and \ref{lem.cube} we have, for any $x\in{Q}$,
$$\big|[|b|,M_{\alpha}](\chi_{Q})(x)\big|
  \le \|b\|_{\dot{\Lambda}_{\beta}}M_{\alpha+\beta}(\chi_{Q})(x)
   \le {C}\|b\|_{\dot{\Lambda}_{\beta}}|Q|^{(\alpha+\beta)/n}
$$
and
$$\big|[|b|,M](\chi_{Q})(x)\big|  \le \|b\|_{\dot{\Lambda}_{\beta}}
  M_{\beta}(\chi_{Q})(x)  \le {C}\|b\|_{\dot{\Lambda}_{\beta}}|Q|^{\beta/n}.
$$

By (\ref{equ.nc-lip-3}) we have
\begin{align*}
I_2 &= \frac1{|Q|^{1+\beta/n}}\int_{Q}
  \big||Q|^{-\alpha/n}M_{\alpha,Q}(b)(x)-M_{Q}(b)(x)\big|dx\\
 &\le \frac{C}{|Q|^{1+(\alpha+\beta)/n}} \int_{Q}
     \big|[|b|,M_{\alpha}](\chi_{Q})(x)\big|dx\\
 &\qquad+\frac{C}{|Q|^{1+\beta/n}} \int_{Q}\big|[|b|,M]
        (\chi_{Q})(x)\big|dx\\
 &\le {C}\|b\|_{\dot{\Lambda}_{\beta}}.
\end{align*}

Putting the above estimates for $I_1$ and $I_2$ into
(\ref{equ.nc-lip-2}) we obtain (\ref{equ.nc-lip-1}).

$(3)\Rightarrow(5)$. Assume statement (3) is true, reasoning as the
proof of $(2)\Rightarrow(4)$, we have
\begin{equation}         \label{equ.nc-lip-4}
\sup_{Q} \frac1{|Q|^{\beta/n}}
\frac{\big\|\big(b-|Q|^{-\alpha/n}M_{\alpha,Q}(b)\big)
\chi_{Q}\big\|_{L^{q(\cdot)}(\rn)}}{\|\chi_{Q}\|_{L^{q(\cdot)}(\rn)}}
<\infty
\end{equation}
for any $q(\cdot)$ satisfying that there exists $p(\cdot)$ such that
$(p(\cdot),q(\cdot))\in \mathscr{B}^{\alpha+\beta}(\rn)$.

For any $s(\cdot)\in\mathscr{B}(\rn)$, choose an $r>n/(n-\beta)$,
we have $rs(\cdot)(n-\beta)/n\in\mathscr{B}(\rn)$ and
$rs(\cdot)\in\mathscr{B}(\rn)$ by Remark \ref{rem.cfmp-1}.
Set $q(\cdot)=rs(\cdot)$ and define $p(\cdot)$ by
$1/{p(\cdot)}=1/{q(\cdot)} +{(\alpha+\beta)}/{n}$. It is easy to
check that $(p(\cdot),q(\cdot))\in \mathscr{B}^{\alpha+\beta}(\rn)$.

Noting that
$$\frac1{s(\cdot)}=\frac1{rs(\cdot)}+\frac1{r's(\cdot)}
=\frac1{q(\cdot)}+\frac1{r's(\cdot)},
$$
it follows from Lemma \ref{lem.holder} (ii), (\ref{equ.nc-lip-4}) and
Lemma \ref{lem.cu-w} that
\begin{align*}
&\frac1{|Q|^{\beta/n}}
\frac{\big\|\big(b-|Q|^{-\alpha/n}M_{\alpha,Q}(b)\big)\chi_Q\big\|_{L^{s(\cdot)}(\rn)}}
{\|\chi_Q\|_{L^{s(\cdot)}(\rn)}}\\
&\le \frac1{|Q|^{\beta/n}}
\frac{\big\|\big(b-|Q|^{-\alpha/n}M_{\alpha,Q}(b)\big)\chi_Q\big\|_{L^{q(\cdot)}(\rn)}
    \|\chi_Q\|_{L^{r's(\cdot)}(\rn)}}
{\|\chi_Q\|_{L^{s(\cdot)}(\rn)}}\\
&\le \frac{C\|\chi_Q\|_{L^{q(\cdot)}(\rn)}\|\chi_Q\|_{L^{r's(\cdot)}(\rn)}}
{\|\chi_Q\|_{L^{s(\cdot)}(\rn)}}\\
&=\frac{C\|\chi_Q\|^{1/r}_{L^{s(\cdot)}(\rn)}\|\chi_Q\|^{1/r'}_{L^{s(\cdot)}(\rn)}}
{\|\chi_Q\|_{L^{s(\cdot)}(\rn)}}=C,
\end{align*}
which is what we want.

The proof of Theorem \ref{thm.nc-lip} is finished.
\end{proof}

\begin{remark} \label{(3)to(5)}
The proof of $(3)\Rightarrow(5)$ is also valid for $\beta=0$.
\end{remark}

To prove Theorem \ref{thm.nc-bmo}, we recall the following results
obtained in \cite{zhw2}.

\begin{lemma}    \label{lem.zhw2-1}
(1) Let $p(\cdot)\in\mathscr{B}(\rn)$. If $0\le {b}\in{BMO(\rn)}$,
then $[b,M]$ is bounded from $L^{p(\cdot)}(\rn)$ to itself.

(2) Let $0<\gamma<n$, $p(\cdot)\in\mathscr{P}(\rn)$ with
$p_{+}<n/\gamma$, $1/q(\cdot)=1/p(\cdot)-\gamma/n$ and
$q(\cdot)/(n-\gamma)\in\mathscr{B}(\rn)$. If $0\le
{b}\in{BMO(\rn)}$, then $[b,M_{\alpha}]$ is bounded from
$L^{p(\cdot)}(\rn)$ to $L^{q(\cdot)}(\rn)$.
\end{lemma}

The following result can be deduced from the proof of Lemma 4.1 in
\cite{zhw2}.

\begin{lemma}    \label{lem.zhw2-2}
Let $0<\gamma<n$. If $b$ is a locally integrable function and
satisfies
$$\sup_{Q}\frac{\|\big(b-|Q|^{-\gamma/n} M_{\gamma,Q}(b)\big)
 \chi_Q\|_{L^{s(\cdot)}(\rn)}}{\|\chi_Q\|_{L^{s(\cdot)}(\rn)}}
  <\infty
$$
for some $s(\cdot)\in\mathscr{B}(\rn)$, then $b\in {BMO(\rn)}$.
\end{lemma}

\begin{proof} {\bf of Theorem \ref{thm.nc-bmo}}~
Since the equivalence of (1), (2) and (3) was given in \cite[Theorem
1.1]{zhw2}, the implication $(2)\Rightarrow(4)$ follows from
\cite[Lemma 4.1]{zhw2} and $(3)\Rightarrow(5)$ follows from Remark
\ref{(3)to(5)}, we only need to prove the implication
$(4)\Rightarrow(1)$.

For any fixed cube $Q$, it follows from (\ref{equ.nc-lip-2}) and
(\ref{equ.nc-lip-3}) that
\begin{equation}      \label{equ.nc-bmo-1} %-------------------------------------------------
\begin{split}
 \frac1{|Q|} \int_{Q}\big|b(x)-M_{Q}(b)(x)\big|dx
 &\le \frac1{|Q|}
  \int_{Q}\big|b(x)-|Q|^{-\alpha/n}M_{\alpha,Q}(b)(x)\big|dx\\
  &\qquad + \frac1{|Q|^{1+\alpha/n}}
    \int_{Q}\big|[|b|,M_{\alpha}](\chi_{Q})(x)\big|dx\\
  &\qquad + \frac1{|Q|}\int_{Q}\big|[|b|,M](\chi_{Q})(x)\big|dx\\
  &:=J_1+J_2+J_3.
\end{split}
\end{equation}

For $J_1$, by Lemma \ref{lem.holder} (i), Lemma \ref{lem.izuki} and
statement (4) we have
\begin{align*}
J_1&\le\frac{C}{|Q|}\big\|\big(b-|Q|^{-\alpha/n}
  M_{\alpha,Q}(b)\big)\chi_Q\big\|_{L^{s(\cdot)}(\rn)}
  \|\chi_Q\|_{L^{s'(\cdot)}(\rn)}\\
 &\le \frac{C\big\|\big(b-|Q|^{-\alpha/n}M_{\alpha,Q}(b)\big)
   \chi_Q\big\|_{L^{s(\cdot)}(\rn)}}
   {\|\chi_Q\|_{L^{s(\cdot)}(\rn)}}\\
 &\le{C},
\end{align*}
where the constant $C$ is independent of $Q$.

Set $q(\cdot)=s(\cdot)n/(n-\alpha)$. By Remark \ref{rem.cfmp-1} we
have $q(\cdot)\in\mathscr{B}(\rn)$ since $s(\cdot)\in\mathscr{B}(\rn)$.
Given $p(\cdot)$ by $1/q(\cdot)=1/p(\cdot)-\alpha/n$,
then $p(\cdot)\in\mathscr{P}(\rn)$ and $p_{+}<n/{\alpha}$.

Noticing that $s(\cdot)\in\mathscr{B}(\rn)$, statement (4) along
with Lemma \ref{lem.zhw2-2} gives $b\in{BMO(\rn)}$, which implies
$|b|\in{BMO(\rn)}$. Thus, we can apply Lemma \ref{lem.zhw2-1} to
$[|b|,M_{\alpha}]$ and $[|b|,M]$ for the pair of exponents
$p(\cdot)$ and $q(\cdot)$ given as above and get
$$\big\|[|b|,M](\chi_{Q})\big\|_{L^{p(\cdot)}(\rn)}
  \le {C}\|\chi_Q\|_{L^{p(\cdot)}(\rn)}
$$
and
$$\big\|[|b|,M_{\alpha}](\chi_{Q})\big\|_{L^{q(\cdot)}(\rn)}
  \le {C}\|\chi_Q\|_{L^{p(\cdot)}(\rn)}.
$$

Then, it follows from Lemma \ref{lem.holder} (i), Lemma
\ref{lem.zhp2} and Lemma \ref{lem.izuki} that
\begin{align*}
J_2&=\frac1{|Q|^{1+\alpha/n}}
  \int_{Q}\big|[|b|,M_{\alpha}](\chi_{Q})(x)\big|dx\\
  &\le\frac{C}{|Q|^{1+\alpha/n}}
    \big\|[|b|,M_{\alpha}](\chi_{Q})\big\|_{L^{q(\cdot)}(\rn)}
     \|\chi_{Q}\|_{L^{q'(\cdot)}(\rn)}\\
  &\le\frac{C}{|Q|^{1+\alpha/n}}
    \|\chi_{Q}\|_{L^{p(\cdot)}(\rn)}
     \|\chi_{Q}\|_{L^{q'(\cdot)}(\rn)}\\
  &\le\frac{C}{|Q|} \|\chi_Q\|_{L^{q(\cdot)}(\rn)}
       \|\chi_{Q}\|_{L^{q'(\cdot)}(\rn)}\\
  &\le {C}.
\end{align*}

Similarly, by Lemma \ref{lem.holder} (i) and Lemma \ref{lem.izuki}, we have
\begin{align*}
J_3&=\frac1{|Q|}\int_{Q}\big|[|b|,M](\chi_{Q})(x)\big|dx\\
  &\le \frac{C}{|Q|}\big\|[|b|,M](\chi_{Q})\big\|_{L^{p(\cdot)}(\rn)}
     \|\chi_{Q}\|_{L^{p'(\cdot)}(\rn)}\\
  &\le \frac{C}{|Q|} \|\chi_{Q}\|_{L^{p(\cdot)}(\rn)}
     \|\chi_{Q}\|_{L^{p'(\cdot)}(\rn)}\\
  &\le {C}.
\end{align*}

Putting the above estimates for $J_1$, $J_2$ and $J_3$ into
(\ref{equ.nc-bmo-1}), we obtain
$$\frac1{|Q|} \int_{Q}\big|b(x)-M_{Q}(b)(x)\big|dx\le{C},
$$
which implies $b\in{BMO(\rn)}$ and $b^{-}\in{L^{\infty}(\rn)}$ by
Lemma \ref{lem.bmr}, since the constant $C$ is independent of $Q$.

The proof of Theorem \ref{thm.nc-bmo} is completed.
\end{proof}

\begin{proof} {\bf of Theorem \ref{thm.mc-lip}}~ Since the
equivalence of (1), (4) and (5) were proved in \cite[Corollary
1.1]{zhp2}, we only need to prove the implications
$(1)\Rightarrow(3)$ and $(2)\Rightarrow(4)$.

$(1)\Rightarrow(3)$. If $b\in {\dot{\Lambda}_{\beta}(\rn)}$, then
\begin{align*}
M_{\alpha,b}(f)(x) &=\sup_{Q\ni {x}}\frac1{|Q|^{1-\alpha/n}}
 \int_Q|b(x)-b(y)||f(y)|dy\\
  &\le {C}\|b\|_{\dot{\Lambda}_{\beta}} \sup_{Q\ni {x}}
    \frac1{|Q|^{1-(\alpha+\beta)/n}} \int_Q |f(y)|dy\\
    &= {C}\|b\|_{\dot{\Lambda}_{\beta}} M_{\alpha+\beta}(f)(x).
\end{align*}
This, together with Lemma \ref{lem.fracmax-var}, shows
$M_{\alpha,b}$ is bounded from $L^{p(\cdot)}(\rn)$ to
$L^{q(\cdot)}(\rn)$.

$(2)\Rightarrow(4)$. For any fixed cube $Q$, we have for all
$x\in{Q}$,
\begin{align*}
|b(x)-b_Q|
 &\le \frac1{|Q|}\int_Q|b(x)-b(y)|dy\\
 &=\frac1{|Q|}\int_Q|b(x)-b(y)|\chi_Q(y)dy\\
 &\le |Q|^{-\alpha/n}{M_{\alpha,b}(\chi_Q)(x)}.
\end{align*}
Then, for all $x\in \rn$,
$$|(b(x)-b_Q)\chi_Q(x)| \le |Q|^{-\alpha/n}{M_{\alpha,b}(\chi_Q)(x)}.
$$

Since $M_{\alpha,b}$ is bounded from $L^{p(\cdot)}(\rn)$ to
$L^{q(\cdot)}(\rn)$, then by Lemma \ref{lem.zhp2} we have
\begin{align*}
\|(b-b_Q)\chi_Q\|_{L^{q(\cdot)}(\rn)}
 &\le |Q|^{-\alpha/n}\|M_{\alpha,b}(\chi_Q)\|_{L^{q(\cdot)}(\rn)}\\
  &\le {C}\|M_{\alpha,b}\|_{L^{p(\cdot)}\to {L^{q(\cdot)}}}
   |Q|^{-\alpha/n}\|\chi_Q\|_{L^{p(\cdot)}(\rn)}\\
  &\le {C}\|M_{\alpha,b}\|_{L^{p(\cdot)}\to {L^{q(\cdot)}}}
    |Q|^{\beta/n}\|\chi_Q\|_{L^{q(\cdot)}(\rn)},
\end{align*}
which gives (\ref{equ.mc-lip}) for $s(\cdot)=q(\cdot)$ since
$Q$ is arbitrary and $C$ is independent of $Q$.

The proof of Theorem \ref{thm.mc-lip} is finished.
\end{proof}

{\bf Acknowledgments}~ The authors are grateful to the referee
for careful reading of the paper and valuable suggestions and comments.

\bigskip
{\bf Funding}

The first author is supported by the National Natural
 Science Foundation of China (Grant Nos. 11571160, 11471176)
  and the Scientific Research Fund of Mudanjiang Normal University (No. MSB201201).
    The second author is supported by the Key Research Project for Higher
     Education in Henan Province (No. 19A110017).

\medskip
{\bf Competing interests}

The authors declare that they have no competing interests.

\medskip
{\bf Authors' contributions}

All authors read and approved the final manuscript.

\end{document}